\documentclass[12pt,a4paper]{amsart}
\usepackage[latin1]{inputenc}
\usepackage[russian,english]{babel}
\usepackage{amsmath,amssymb,amsthm}

\usepackage{enumerate}
\usepackage{cite}

 \theoremstyle{plain}
\newtheorem{theorem}{Theorem}[section]
\newtheorem{lemma}[theorem]{Lemma}
\newtheorem{corollary}[theorem]{Corollary}
\newtheorem{prop}[theorem]{Proposition}

\theoremstyle{definition}
\newtheorem{definition}[theorem]{Definition}
\newtheorem{example}[theorem]{Example}

\newcommand{\abs}[1]{\lvert#1\rvert}

\begin{document}

\title{Isometric Isomorphisms of the annihilator of $C_0(G)$ in $LUC(G)^*$}



\author[S. Zadeh]{Safoura Zadeh}

\address{Department of Mathematics\\
342 Machray Hall, 186 Dysart Road\\
University of Manitoba\\
Winnipeg, MB R3T 2N2\\
 Canada.
} 
\email{\text{jsafoora@gmail.com}}

\begin{abstract}
Let $LUC(G)$ denote the $C^*$-algebra of left uniformly continuous functions with the uniform norm and let $C_0(G)^{\perp}$ denote the annihilator of $C_0(G)$ in $LUC(G)^*$. In this article, among other results, we show that if $G$ is a locally compact group and $H$ is a discrete group then whenever there exists a weak-star continuous isometric isomorphism between $C_0(G)^{\perp}$ and $C_0(H)^{\perp}$, $G$ is isomorphic to $H$ as a topological group. In particular, when $H$ is discrete $C_0(H)^{\perp}$ determines $H$ within the class of locally compact topological groups.
\end{abstract}
\maketitle

\begin{section}{Introduction and preliminaries}
In \cite{MR0049910} Wendel showed that two locally compact groups are isomorphic if their group algebras are isometrically isomorphic. Similar results for different algebras have since been proved. Johnson \cite{MR0160846}  proved that the same conclusion holds when the group algebra is replaced by a measure algebra. Lau and McKennon in \cite{MR560583} generalized Johnson's result. They gave a direct proof to show that as a Banach algebra with its left Arens product, the dual of a left introverted subspace of $C_b(G)$ that contains $C_0(G)$ determines $G$. In particular, it can be concluded that the dual of the left uniformly continuous functions on a locally compact group $G$, $LUC(G)^*$, determines $G$. Their result also implies Johnson's aforementioned result. In \cite{MR1005079} Ghahramani, Lau and Losert showed that the second dual of the group algebra also determines the underlying group. Recently, Dales, Lau and Strauss \cite{MR2920625} proved that $M(G)^{**}$, determines the locally compact group $G$. In this article among other results we prove that when $G$ is discrete $C_0(G)^{\perp}$, the annihilator of $C_0(G)$ in $LUC(G)^*$, determines $G$ as a locally compact group.\\

Let $G$ be a Hausdorff locally compact group and $C_b(G)$ denote the Banach space of bounded continuous functions on $G$ with the uniform norm. An element $f$ in $C_b(G)$ is called left uniformly continuous if $g\mapsto l_gf$, where $l_gf(x)=f(gx)$, is a continuous map from $G$ to $C_b(G)$. Let $LUC(G)$ denote the Banach space of left uniformly continuous functions with the uniform norm. The dual space of $LUC(G)$, $LUC(G)^*$, with the following Arens-type product forms a Banach algebra

\begin{center}
$ \langle m.n,f\rangle=\langle m,nf\rangle\ \ \ {\rm and} \ \ \   nf(g) = \langle n,l_g(f)\rangle$
\end{center}
\noindent
where $m,n\in LUC(G)^*$, $f\in LUC(G)$ and $g\in G.$ The measure algebra $M(G)$ can be identified with a subalgebra of $LUC(G)^*$ via 
$$\langle \mu,f\rangle=\int_G f\ d\mu\ \ \ \left(\mu\in M(G),\ f\in LUC(G)\right).$$
In fact we have the $L^1$- direct sum
$$LUC(G)^*=M(G)\oplus_1 C_0(G)^{\perp}$$
 where
$$C_0(G)^{\perp}:=\{m\in LUC(G)^*: \langle m,f\rangle=0,\ \forall f\in C_0(G)\}.$$
Moreover, $C_0(G)^{\perp}$ is a weak-star closed ideal in $LUC(G)^*$ (see \cite[Lemma 1.1]{MR1005079}). One can easily check that for each $n\in LUC(G)^*$ and $\mu\in M(G)$ the mappings $m\mapsto m\cdot n$ and $n\mapsto \mu\cdot n$ are weak-star continuous. In fact, $M(G)$ is the largest subset of elements $m$ in $LUC(G)^*$ for which the left multiplication mapping $n\mapsto m\cdot n,\ m\in LUC(G)^*$ is continuous (see \cite{MR817669}).\\

The left uniformly continuous compactification of $G$, $G^{luc}$, is the Gelfand spectrum of the unital commutative $C^*$-  algebra $LUC(G)$, that is
$$G^{luc}:=\{m\in LUC(G)^*\setminus\{0\};\langle m,fg\rangle=\langle m,f\rangle\langle m,g\rangle\ \forall f,g\in LUC(G)\}.$$
It can be shown that $G^{luc}$ is in fact a weak-star compact semigroup with the Arens multiplication and weak-star topology it inherits from $LUC(G)^*$. When $G$ is a discrete group, $LUC(G)=l^{\infty}(G)=C(\beta G)$ and so the $LUC$-compactification of $G$ is the same as its Stone-\v{C}ech compactification. We identify $G$ with its image in $G^{luc}$. The corona of the $LUC$-compactification of $G$, $G^{luc}\setminus G$, is denoted by $G^*$ and is a closed ideal of the compact semigroup $G^{luc}$. An element $z\in G^*$ is called right cancellable if for each $m,n\in G^{luc}$, $mz=nz$ implies that $m=n$. In \cite[Thm. 1]{MR1935082}, it is proved that the $LUC$- compactification contains many right cancellable elements.\\

Let $G$ and $H$ be locally compact groups and $\mathbb{T}$ be the circle group. Suppose that $\alpha:G\to\mathbb{T}$ is a continuous character and $\psi:G\to H$ is a continuous homomorphism.  Then it is easy to see that 
$$j_{\alpha,\psi}(f):=\alpha\cdot f\circ \psi$$
maps $C_0(H)$ into $C_0(G)$ and the dual mapping
$$j_{\alpha,\psi}^*:M(G)\to M(H)$$
is a homomorphism. When $\psi$ is an isomorphism $j_{\alpha,\psi}^*:M(G)\to M(H)$ is a weak-star continuous isometric isomorphism. It follows from \cite{MR0160846} that every isometric isomorphism $T:M(G)\to M(H)$ is of the form $T=j_{\alpha,\psi}^*$, for some character $\alpha:G\to\mathbb{T}$ and isomorphism $\psi:G\to H$, and therefore is weak-star continuous. Similarly, if $\alpha:G\to\mathbb{T}$ is a continuous character and $\psi:G\to H$ is a continuous homomorphism, then it is easy to see that 
$$j_{\alpha,\psi}(f):=\alpha\cdot f\circ\psi$$ maps $LUC(H)$ into $LUC(G)$ and that the dual map $$j_{\alpha,\psi}^*:LUC(G)^*\to LUC(H)^*$$ is a homomorphism. When $\psi$ is a topological isomorphism $j_{\alpha,\psi}^*:LUC(G)^*\to LUC(H)^*$ is a weak-star continuous isometric isomorphism. Moreover, every weak-star continuous isometric isomorphism $T:LUC(G)^*\to LUC(H)^*$ also takes this canonical form, but it is not clear if every isometric isomorphism $T:LUC(G)^*\to LUC(H)^*$ is weak-star continuous. \\
\indent Suppose that $\psi:G\to H$  is a continuous homomorphism. Letting $\tilde{\psi}$ denote the restriction of $j_{1,\psi}^*$ to $G^{luc}$, $\tilde{\psi}:G^{luc}\to H^{luc}$ is the unique continuous homomorphism extending $\psi$. We note that when $\psi$ is a topological isomorphism, $\tilde{\psi}$ is a topological isomorphism of $G^{luc}$ onto $H^{luc}$.\\

F-spaces were studied in detail by L. Gillman and M. Henriksen in 1956 \cite{MR0078980} as the class of spaces for which $C(X)$ is a ring in which every finitely generated ideal is a principal ideal. Several conditions both topological and algebraic were proved equivalent for a space to be an F-space (see  \cite{MR0407579}). We choose the following characterization as our definition for an F-space.
\begin{definition}
A completely regular space $X$ is an F-space if for any continuous bounded function $f$ on $X$ there is a continuous bounded function $k$ on $X$ such that $f=k \abs{f}.$
\end{definition}
Many of the proofs in the case of discrete groups use the fact that for a discrete space the Stone-\v{C}ech compactification and its corona are F-spaces (see for example \cite{MR1761431}, \cite{MR1429649}, \cite{MR1967820}, \cite{MR1391957}, \cite{MR1297310} and \cite{MR2893605}). This is especially useful due to the following lemma. A proof is given in \cite[Lemma 1.1]{MR1297310}.

\begin{lemma}\label{FS}
If $X$ is a compact space then $X$ is an F-space if and only if for $\sigma-$compact subsets $A$ and $B$ of $X$, $\bar{A}\cap B=\emptyset$ and $A\cap\bar{B}=\emptyset$ implies that $\bar{A}\cap\bar{B}=\emptyset.$
\end{lemma}

In the first section of this article we prove that for a locally compact non-discrete group neither the $LUC$-compactification  nor its corona is an F-space. So in fact, for locally compact groups the corona of the $LUC$-compactification is an F-space if and only if the group is discrete. This shows that some interesting facts concerning Stone--\v{C}ech compactifications of discrete groups cannot possibly be generalized to $LUC$-compactifications of  general locally compact groups following the same line of proof. We also show that the corona of the $LUC$-compactification of a locally compact group does not contain $P$-points. A definition of $P$-points and a short discussion about their importance is given in the next section. In Section 3 we prove that a discrete group $G$ is completely determined within the class of all locally compact groups by both $C_0(G)^{\perp}$ and $G^*$. Some related results are also obtained.
\end{section}
\begin{section}{When is $G^*$ an F-space?}
Our main goal in this section is to show that if $G$ is a locally compact non-discrete group then neither $G^{luc}$ nor $G^*$ are F-spaces. Therefore $G^{luc}$ ($G^*$) is an F-space if and only if $G$ is discrete. This result will be applied in Section 3.\\

In the proof of Theorem \ref{F-space} we make use of absolutely convergent series with alternating partial sums. An example of such a series is given below.

\begin{example}\label{al}
Let $\sum_{n=1}^{\infty} b_n$ be a convergent series with positive terms and let
\begin{align*}
a_1&:=b_1\\
 a_2&:= -b_1-b_2\\
a_3&:= b_2+b_3\\
a_4&:=-b_3-b_4\\
\dots\\
a_{2k}&:=-b_{2k-1}-b_{2k}\\
a_{2k+1}&:=b_{2k}+b_{2k+1}\\
\dots.
\end{align*}

\noindent Then the series $\sum_{n=1}^{\infty} a_n$ is an example of an absolutely convergent series whose partial sums have alternating sign.
\end{example}
\begin{theorem}\label{F-space}
If $X$ is a Hausdorff locally compact non-discrete topological space with a non-trivial convergent sequence, then $X$ is not an F-space.
\end{theorem}
\begin{proof}
Suppose that $X$ contains a non-trivial sequence $(x_n)$ convergent to a (non-isolated) point  $x_0$ in $X$.  Without loss of generality assume that $x_n\neq x_0$ for each $n\in\mathbb{N}$. Since $x_0$ is not isolated and $X$ is locally compact we can inductively construct a nested family $\{K_n\}_{n\geq 2}$ of compact neighbourhoods around $x_0$ such that $x_1,x_2,..., x_n\notin K_{n+1}$ and $x_{(n+1)},x_{(n+2)},...\in K_{n+1}$, for each $n$. To see this note that since $x_1\neq x_0$ there is a precompact open set $U_0$ such that $\{x_2,x_3,...\}\cup \{x_0\} \in U_0$ and $x_1\not\in U_0$. Since $X$ is a locally compact space for  the compact set $\{x_0\}$ and the open set $U_0$ there is a compact set $K_2$ such that $\{x_0\}\subseteq K_2^{\circ}\subseteq K_2\subseteq U_0$, where $K_2^{\circ}$ denotes the interior of the set $K_2$. Without loss of generality we can assume that $(x_n)_{n\geq 2}\in K_2$. Suppose that the compact set $K_n$ is given such that $x_1,...,x_{n-1}\not\in K_n$. Since $x_1,...,x_n\neq x_0$ there is a precompact open set $U_n$ such that $\{x_{n+1},x_{n+2},...\}\cup\{x_0\}\subseteq U_n\subseteq K_n$ and $x_1,...,x_n\not\in U_n$.  Since $X$ is locally compact there is a compact set $K_{n+1}$ such that $\{x_0\}\subseteq K_{n+1}^{\circ}\subseteq K_{n+1}\subseteq U_n$ and note that $x_1,...,x_n\not\in K_{n+1}$. By discarding some elements of our sequence if necessary we can assume that $x_{(n+1)},x_{(n+2)},...\in K_{n+1}$.\\
\indent Consider an absolutely convergent series $\sum_{n=1}^{\infty}a_n$, with alternating partial sums as in Example \ref{al}. Using Urysohn's lemma for locally compact spaces, for $n\geq 2$ we define the compactly supported function $f_n$ such that  $f_n(X\setminus K_n)=0$ and $f_n(k_{n+1}\cup \{x_n\})=a_{n-1}$, for $n\geq 2$. Let $f:=\sum_{n\geq2} f_n$. By the Weierstrass M-test, the series $\sum_{n\geq2} f_n$ is uniformly convergent to a continuous function. Note that $f(X\setminus K_2)=0$. For any function $k$ where $f=k \abs{f}$, we observe that $k$ is not continuous at $x_0$. To see this we note that $f(x_n)=\sum_{m=1}^n a_n$, $n\geq2$. Therefore, $f(x_n)f(x_{n+1})<0$ as $\sum_{m=1}^n a_m\sum_{m=1}^{n+1} a_{m}<0$. So the function $k$ alternates on the convergent sequence $(x_n)$ between $+1$ and $-1$ and it is not continuous at $x_0$.
\end{proof}
So to show that $G^{luc}$ ($G^*$) is not an F-space, it is enough to show that it contains a non-trivial convergent sequence. 

Kuzminov has shown that any compact group is dyadic, i.e., a continuous image of a Cantor cube. This implies that every infinite compact group contains a non-trivial convergent sequence. This result can be found for example in J. Van Mill's article in \cite{MR2367385} (see page 190).

\begin{theorem}\label{not F-space}
If $G$ is a locally compact non-discrete group then $G$ has a non-trivial convergent sequence. In particular, $G$ is not an $F-$space.
\end{theorem}
\begin{proof}
It is easy to see that if $G$ is a metrizable non-discrete locally compact group then $G$ has a non-trivial convergent sequence. If $G$ is a locally compact non-discrete group then $G$ has a subgroup $H$ that is sigma-compact, clopen and non-discrete. If $H$ is metrizable then any point in $H$ is a limit point of a non-trivial convergent sequence in $H$. This sequence is also convergent in $G$. So suppose that $H$ is not metrizable. By the Kakutani-Kodaira theorem \cite[Thm. 8.5]{MR551496} there is a compact normal subgroup $N$ of $H$ such that $H/N$ is metrizable, and so by \cite[ 5.38 part (e)]{MR551496}, $N$ cannot be metrizable since otherwise $H$ would be metrizable. Because $N$ is not metrizable it cannot be finite. As noted above, since $N$ is compact and non-discrete we can find a non-trivial convergent sequence in $N$. Any such non-trivial convergent sequence is also convergent in the open subgroup $H$ and therefore in $G$.
\end{proof} 
As stated in the introduction, were $G^*$ an F-space for some non-discrete groups, we would be able to prove stronger versions of our main results in the next section namely, Theorem \ref{nice} and Corollary \ref{abc}. However, as  the next result shows, $G^*$ is an F-space if and only if $G$ is discrete.
\begin{theorem}\label{F}
Suppose that $G$ is a locally compact non-discrete group. Then neither $G^{luc}$ nor $G^*$ is an F-space.
\end{theorem} 

\begin{proof}
Suppose that $G$ is a locally compact non-discrete group. Then $G$ contains a non-trivial sequence, say $(x_n)$, convergent to an element $x$. Let $z$ be a right cancellable element in $G^*$. Then the non-trivial sequence $(x_nz)$ converges to $xz$. So in both cases, neither $G^{luc}$ nor $G^*$ is an F-space, by Theorem \ref{F-space}.
\end{proof}

\begin{definition}
A point in a topological space is called a $P$-point if every $G_{\delta}-$set containing the point is a neighbourhood of the point.
\end{definition}
 Gillman and Henriksen \cite{MR0407579} were the first to study $P$-points. If $G$ is a discrete group then under the continuum hypothesis the set of $P$-points in $G^*$ forms a dense subset in $G^*$. It is a fact that the existence of $P$-points cannot be proved in $Z F C$. We refer the reader to \cite{MR2893605} and the remark in \cite[page 385]{MR1377702} for an explanation of these statements. From Corollary \ref{p} (below) we see that under the continuum hypothesis, if $G$ is a locally compact group, then $G^{luc}\setminus G$ has a $P$-point if and only if $G$ is discrete. $P$-points were used in \cite[Thm. 3]{MR1429649} to show that for a discrete group $G$ there are left cancellable elements. In fact these particular $P$-points in \cite[Thm. 3]{MR1429649} are also right cancellable. It is not known if for the general case of discrete groups there can be left cancellable elements that are not right cancellable (see \cite[Thm. 8.40]{MR2893605}).
\begin{corollary}\label{p}
Suppose that $G$ is a locally compact non-discrete group. Then $G^*$ does not contain any $P$-point.
\end{corollary}
\begin{proof}
First observe that for each point $p$ in $G^*$ there is a non-trivial sequence $(x_n)$ converging to $p$. To see this, note that from the proof of Lemma \ref{not F-space} there is a non-trivial sequence $(z_n)$, convergent to a point $y$ in the group. Therefore, the non-trivial sequence $(y_n)$ where $y_n:=y^{-1}z_n$ is convergent to $e$. Now $x_n:=y_np$ is a non-trivial sequence convergent to $p$ in $G^*$. Let $\{U_n\}$ be a family of open neighbourhoods of $p$ such that for each $n$, $x_1,x_2,...,x_n\not\in U_n$. Then $p\in \cap_m U_m$, but for each $n$, $x_n\not\in\cap_m U_m$ so $\cap_m U_m$ cannot be open. Hence $p$ is not a $P$-point.
\end{proof}

\end{section}
\begin{section}{Isomorphism on $C_0(G)^{\perp}$}

In \cite{MR1005079} Ghahramani, Lau and Losert showed that given locally compact groups $G$ and $H$, every isometric isomorphism $\displaystyle{T:LUC(G)^*\to LUC(H)^*}$ maps $M(G)$ onto $M(H)$. It can be shown directly that $G^{luc}$ determines $G$, within the class of locally compact groups. To see this we note that $G^*$ is an ideal in $G^{luc}$ and therefore the only invertible elements of $G^{luc}$ are elements of $G$. An interesting question is wether $C_0(G)^{\perp}$ ($G^*$) also determines $G$. In this section we employ Theorem \ref{F} to show that this is the case when $G$ is discrete (Theorem \ref{nice} and Corollary \ref{abc}).\\

 First we show that if $G$ and $H$ are non-compact locally compact groups and if $T:LUC(G)^*\to LUC(H)^*$ is an isometric isomorphism, then $C_0(G)^{\perp}$ is also mapped onto $C_0(H)^{\perp}$. We recall the following lemma from \cite{MR0177058}.

\begin{lemma}\label{orthogonal}
Let $X$ be a locally compact space, and let $\mu\text{ and }\nu \in M (X)$. Then $\mu$ and $\nu$ are mutually singular if
and only if $\|\mu + \nu\| =\| \mu -\nu\| = \|\mu\| +\|\nu \|$.
\end{lemma}
Lemma \ref{orthogonal} implies that isometries preserve mutual singularity. \\

Since $LUC(G)=C(G^{luc})$, the Banach space of continuous functions on $G^{luc}$, we have that $LUC(G)^*=M(G^{luc})$, the Banach space of all regular Borel measures on $G^{luc}$. It can be seen that $C_0(G)^{\perp}$ is isometrically isomorphic to $M(G^*)$, the Banach space of regular Borel measures on $G^*$. Therefore, $M(G^{luc})=LUC(G)^*= M(G)\oplus_1 C_0(G)^{\perp}=M(G)\oplus_1 M(G^*)$.
\begin{theorem}
Suppose that $G$ and $H$ are non-compact locally compact groups. If $T: LU C(G)^{*}\to LU C(H)^{*}$ is an isometric isomorphism. Then $T$ maps $C_0 (G)^{\perp}$ onto $C_0 (H)^{\perp}$. 
\end{theorem} 
\begin{proof}
Let $m \in C_0 (G)^{\perp}$ , we show that $T(m) \in C_0 (H)^{\perp}$ . Suppose that $T(m) = \nu + r$ , where $\nu \in M (H)$ and $r \in C_0 (H)^{\perp}$, so that $\nu$ and $r$ are mutually singular measures in $M (H^{luc})$. Since an isometry preserves mutual singularity, if we consider $m =T^{-1} (\nu) + T^{-1} (r)$, then we have that $T^{-1} (\nu)$ and $T^{-1} (r)$ are singular and also we have that $T^{-1} (\nu) \in M (G)$, by \cite[Thm. 1.6]{MR1005079}. Suppose that $T^{-1} (r) = \nu^{\prime} + r^{\prime}$ where $\nu^{\prime} \in M (G)$ and $r^{\prime} \in C_0 (G)^{\perp}.$ Therefore $m = T^{-1} (\nu) + \nu^{\prime} + r^{\prime}.$ Now, since $T^{-1} (\nu), \nu^{\prime} \in M (G)$ we must have $T^{-1} (\nu) + \nu^{\prime} = 0$ because $M (G) \cap C_0 (G)^{\perp} = 0.$ Since $\nu^{\prime}$ is absolutely continuous with respect to $T^{-1}(r)$, it is mutually singular with $T^{-1}(\nu)$, and so we have that $T^{-1} (\nu) = 0$. Hence $T(m)=\nu+r=r\in C_0(G)^{\perp}$. 
\end{proof}
\begin{corollary}\label{a}
Suppose that $G$ and $H$ are locally compact groups. Then $T:LUC(G)^*\to LUC(H)^*$ is an isometric isomorphism if and only if there are isometric isomorphisms $T_1:M(G)\to M(H)$ and $T_2:C_0(G)^{\perp}\to C_0(H)^{\perp}$ such that $T=T_1+T_2$, $T_2(\mu m)=T_1(\mu)T_2(m)$ and $T_2(m\mu)=T_2(m)T_1(\mu)$ for all $m\in C_0(G)^{\perp}$ and $\mu\in M(G)$.
\end{corollary}
Note that Corollary \ref{a} shows that the value of $T_1$ is connected to that of $T_2$. It is  not clear if every isometric isomorphism $T:C_0(G)^{\perp}\to C_0(H)^{\perp}$ can be extended to one on $LUC(G)^*$. When either $G$ or $H$ is abelian and discrete and $T$ is weak-star continuous, we will show that such an extension always exists (Theorem \ref{extension}). Proposition \ref{unique} shows that such an extension is always unique. We need the following lemma for the proof of Proposition \ref{unique}.
\begin{lemma}\label{two parts}
Suppose that $G$ is a locally compact group and let $z\in G^*$ be a right cancellable in $G^{luc}$. The following statements hold:\\
$(i)$ $z$ is also right cancellable in $LUC(G)^*$.\\
$(ii)$ If $(n_i)$ is a bounded net in $LUC(G)^*$ and $n_iz \xrightarrow{w^*} nz$ in $LUC(G)^*$, then $n_i\xrightarrow{w^*} n$ in $LUC(G)^*$. 
\end{lemma}
\begin{proof}
$(i)$ Let $z\in G^*$ be right cancellable in $G^{luc}$. Then we show that $z$ is also right cancellable in $LUC(G)^*$. To see this we note that since $z$ is right cancellable, the unital $*$-algebra $$\{zf,f\in LUC(G)\}\ \ \ \text{where }\ \  zf(x)=\langle z, l_xf\rangle$$ separates the points in $LUC(G)$ and thus, by the Stone-Weierstrass theorem, is dense in $LUC(G)=C(G^{luc})$. So if for some $m,n\in LUC(G)^*$ we have that $mz=nz$, then for all $f\in LUC(G)$ we must have $\langle m,zf\rangle=\langle n,zf\rangle$ and therefore $m=n$.\\
$(ii)$ Suppose that $(n_i)$ is a net in $LUC(G)^*$ such that $(n_i)$ is bounded in norm by $M>0$ and $n_iz \xrightarrow{w^*} nz$ in $LUC(G)^*$. We show that $n_i\xrightarrow{w^*}n$ in $LUC(G)^*$. Suppose that $\varepsilon>0$ and $f\in LUC(G)$ are given. As noted in the proof of part $(i)$, the algebra $\{z g,\ g\in LUC(G)\}$ is norm-dense in $LUC(G)$ so there is $g\in LUC(G)$ such that 
$$\|f-zg\|\leq\frac{\varepsilon}{M}.$$ 
Also $n_iz\xrightarrow{w^*}n z$, so there is $i_0$ such that for all $i\geq i_0$ we have
$$\abs{ n_iz(g)-nz(g)} \leq\varepsilon .$$
Therefore for all $i\geq i_o$,
\begin{eqnarray*}
\abs{  n_i (f)-n(f)} & \leq& \abs{ n_i(f)-n_i(zg)} +\abs{ n_i(zg)-n(zg)} \\
& +& \abs{ n(zg)-n(f)}  \\
& \leq& \varepsilon + \varepsilon+ \varepsilon.
\end{eqnarray*}
\end{proof}
\begin{prop}\label{unique}
Suppose that $G$ and $H$ are locally compact groups and $T:C_0(G)^{\perp}\to C_0(H)^{\perp}$ is an (algebraic) isomorphism. If there is an isomorphism $\tilde{T}:LUC(G)^*\to LUC(H)^*$ such that $\tilde{T}$ is an extension of $T$, then $\tilde{T}$ is unique.
\end{prop}
\begin{proof}
Suppose that $T_1$ and $T_2$ are two such extensions. Let $z\in G^*$ be right cancellable in $G^{luc}$. Then, by Lemma \ref{two parts}, $z$ is also right cancellable in $LUC(G)^*$ and hence, $T(z)=T_1(z)=T_2(z)$ is also right cancellable in $LUC(H)^*$. For each $m$ in $LUC(G)^*$, then we have that
$$T_1(m)T(z)=T(mz)=T_2(m)T(z).$$ Since $T(z)$ is right cancellable we have that $T_1(m)=T_2(m)$, for all $m$ in $LUC(G)^*$.
\end{proof}

The following proposition shows that under certain conditions the extension of an algebraic homomorphism $\phi:G^*\to H^*$ to an algebraic homomorphism $\varphi:G^{luc}\to H^{luc}$ is also unique.
\begin{prop}\label{ext}
Suppose that $G$ and $H$ are locally compact groups and $\phi:G^*\to H^*$ is an algebraic homomorphism. Suppose that $\varphi_L,\varphi_{L^{\prime}}:G^{luc}\to H^{luc}$ are homomorphic extensions of $\phi$ such that either\\
$(i)\ \varphi_{L}(G),\varphi_{L^{\prime}}(G)\subseteq H,\text{or}$\\ 
$(ii)$ the interior of $\phi(G^*)$ is non-empty.\\
Then, $\varphi_{L}=\varphi_{L^{\prime}}.$
\end{prop}
\begin{proof}
Let $x\in G$. Then, for all $p\in G^*$,
 $$\varphi_{L}(x)\phi(p)=\varphi_{L}(xp)=\phi(xp)=\varphi_{L^{\prime}}(xp)=\varphi_{L^{\prime}}(x)\phi(p).$$ In the case of $(i)$, $\varphi_{L}(x)=\varphi_{L^{\prime}}(x)$, by Veech's theorem \cite{MR1847282}. By \cite[Thm. 1]{MR1935082}, the set of right cancellable elements is dense in $G^*$, so in the second case we can choose $p\in G^*$ such that $\phi(p)$ is right cancellable.
\end{proof} 
In fact, if $\varphi:G^{luc}\to H^{luc}$ is a surjective homomorphism, then $\varphi$ satisfies condition $(i)$ in Proposition \ref{ext}. To see this, we observe that since $\varphi$ is onto, there is a $q\in G^{luc}$, such that $e_H=\varphi(q)$. We have that
$$e_H=\varphi(e_Gq)=\varphi(e_G)\varphi(q)=\varphi(e_G)e_H=\varphi(e_G).$$
\noindent Thus, for each $x$ in $G$, we have that $e_H=\varphi(x)\varphi({x^{-1}})$ and thus $\varphi(x)$ must belong to $H$.\\
\indent As noted in the introduction, when $\psi:G\to H$ is a continuous homomorphism and $\alpha:G\to\mathbb{T}$ is the constant character $1$, $j_{1,\psi}^*:LUC(G)^*\to LUC(H)^*$ is a weak-star continuous homomorphism and $\tilde{\psi}:=j_{1,\psi}^*|_{G^{luc}}$ is a continuous homomorphism of $G^{luc}$ into $H^{luc}$. We say that the continuous homomorphism $\phi:G^*\to H^*$ is induced by a continuous homomorphism $\psi:G\to H$, if $\tilde{\psi}|_{G^*}=\phi$.\\
\indent Note that if $\phi:G^*\to H^*$ is a continuous homomorphism, it is not necessarily induced by a homomorphism $\psi:G\to H$. A simple example is $\phi :G^*\to H^*$ where $\phi (p)=\iota$ and $\iota$ is an idempotent  in $H^*$ (see \cite[Thm. 2.5]{MR2893605} for a proof of the existence of idempotents in $G^*$). To see this suppose that there is a homomorphism $\psi:G\to H$ such that $\tilde{\psi}|_{G^*}=\phi$. Then since
 $$\iota=\tilde{\psi}(xp)=\psi(x)\phi(p)=\psi(x)\iota$$ for (any) $p\in G^*$ and $x\in G$, we have that $\psi$ is the trivial homomorphism $x\mapsto e_G$, by Veech's theorem \cite{MR1847282}. By uniqueness, $\tilde{\psi}(p)=e_G$, for all $p\in G^{luc}$, which is not possible as $\tilde{\psi}|_{G^*}=\phi$.\\
\indent We apply Theorem \ref{F} and  \cite[Thm. 6.2]{MR1761431} to prove the next result.
\begin{theorem}\label{nice}
Let $G$ and $H$ be locally compact groups. Suppose that $\phi: G^*\to H^*$ is a continuous isomorphism and that either $G$ or $H$ is discrete. Then there is a unique topological isomorphism $\psi:G\to H$ such that $\phi=\tilde{\psi}|_{G^*}$. In particular, $G$ and $H$ are isomorphic, as topological groups.
\end{theorem}
\begin{proof}
Since $\phi: G^*\to H^*$ is a continuous isomorphism, by Theorem \ref{F} we have that both $G$ and $H$ are discrete groups. The existence of such a continuous surjection $\psi:G\to H$ now follows from \cite[Thm. 6.2]{MR1761431}. We shall also show that $\psi$ is injective. Suppose that for some $x,y\in G$, we have that $\psi(x)=\psi(y)$. Then for each $q\in G^*$, since $H^*$ is an ideal in $H^{luc}$, we have that $$\phi(xq)=\tilde{\psi}(xq)=\psi(x)\phi(q)=\psi(y)\phi(q)=\tilde{\psi}(yq)=\phi(yq).$$ Since $\phi$ is injective, $xq=yq$ and so, by Veech's theorem \cite{MR1847282}, $x=y$. Thus $\psi$ is an isomorphism of topological groups and therefore $G$ and $H$ must be isomorphic.
\end{proof}

Is \cite[Thm. 6.2]{MR1761431} true when $G$ and $H$ are not assumed to be discrete? Although some details such as the existence of prime elements in the proof of \cite[Thm. 6.2]{MR1761431} remain valid for non-discrete locally compact groups, the proof heavily depends on the Lemma \ref{FS} above. By Theorem \ref{F}, Lemma \ref{FS} cannot be employed for non-discrete locally compact groups, so the same proof will not work.

 \begin{corollary}\label{abc}
Let $G$ and $H$ be locally compact groups. Suppose that $T:C_0(G)^{\perp}\to C_0(H)^{\perp}$ is a weak-star continuous isometric isomorphism and either $G$ or $H$ is discrete. Then the topological groups $G$ and $H$ are isomorphic.
 \end{corollary}
 \begin{proof}
 First we note that since $C_0(G)^{\perp}=M(G^*)$, by \cite[Thm. V.8.4]{MR1070713} the set of extreme points of the unit ball of $C_0(G)^{\perp}$ is 
 $$\{\alpha \delta_p;\ \alpha\in\mathbb{T},\ p\in G^*\}.$$
 In particular the points in $G^*$ are among the extreme points of the unit ball of $C_0(G)^{\perp}$. Because $T$ is an isometry, it maps the extreme points of the unit ball of $C_0(G)^{\perp}$ onto the extreme points of the unit ball of $C_0(H)^{\perp}$. Therefore there exist maps $\phi:G^*\to H^*$ and $\alpha:G^*\to\mathbb{T}$ such that $T(\delta_p)=\alpha(p)\delta_{\phi(p)}$, for all $p\in G^*$. We first show that $\alpha$ is continuous. Suppose that $(p_{\gamma})$ is a net in $G^*$ that is convergent to $p\in G^*$. Since $T$ is weak-star continuous we have that 
 $$\alpha(p_{\gamma})\delta_{\phi(p_{\gamma})}=T(\delta_{p_{\gamma}})\xrightarrow{w^*} T(\delta_p)=\alpha(p)\delta_{\phi(p)}.$$ 
Evaluating this equation at $1_{G^*}$ implies that $\alpha(p_{\gamma})\to\alpha(p)$. Hence $\alpha$ is continuous. Similarly, since $T$ is an isomorphism we can show that $\alpha$ is also multiplicative and thus $\alpha$ is in fact a continuous character. Therefore $\phi=\bar{\alpha}T|_{G^*}:G^*\to H^*$ is a continuous isomorphism. Since $G^*$ is compact and $H^*$ is Hausdorff, the continuous isomorphism $\phi:G^*\to H^*$ is also a homeomorphism. Now the result follows from Theorem \ref{nice} .
  \end{proof}
When $G$ and $H$ are both non-discrete locally compact groups, it remains open wether $C_0(G)^{\perp}$ determines $G$. \\

\begin{prop}\label{character0}
Let $G$ be a locally compact group, $K$ a compact group, and suppose that $\alpha:G^*\to K$ is a continuous homomorphism. Let $\iota$ be any idempotent in $G^*$ and define $\alpha_l:G^{luc}\to K$ by $\alpha_l(x)=\alpha(x\iota)$. Then the following statements hold:\\
$(i)$ $\alpha_l$ is a continuous extension of $\alpha$ to $G^{luc}$.\\
$(ii)$ If $\beta$ is any homomorphic extension of $\alpha$ to $G^{luc}$, then $\beta=\alpha_l$. (Thus any such homomorphic extension -if it exists- is unique and automatically continuous.)\\
$(iii)$ If $\iota$ commutes with elements of $G$, then $\alpha_l$ is a homomorphism on $G^{luc}$. 
\end{prop}
\begin{proof}
$(i)$ We note that 
$$\alpha(\iota)=\alpha(\iota^2)=\alpha(\iota)^2,$$
so $\alpha(\iota)=e_K$, the identity in $K$. As $x\mapsto x\iota$ is a continuous mapping of $G^{luc}$ into $G^*$, and $\alpha$ is continuous on $G^*$, $\alpha_l$ is continuous on $G^{luc}$.\\
$(ii)$ If $\beta$ is any such homomorphic extension of $\alpha$, then for any $x\in G^{luc}$, 
$$\beta(x)=\beta(x)e_K=\beta(x)\beta(\iota)=\beta(x\iota)=\alpha(x\iota)=\alpha_l(x).$$
$(iii)$ For any $x,y\in G$,
$$\alpha_l(xy)=\alpha(xy\iota)=\alpha(xy\iota^2)=\alpha(x\iota y\iota)=\alpha(x\iota)\alpha(y\iota)=\alpha_l(x)\alpha_l(y).$$
Thus, $\alpha_G:=\alpha_l|_G:G\to K$ is a continuous homomorphism. (Note that it is not yet clear that $\alpha_l$ is a homomorphism on $G^{luc}$, even though its restrictions to $G$ and $G^*$ are homomorphisms.) As observed in the introduction, $\alpha_G$ extends to a continuous homomorphism $\widetilde{\alpha_G}:G^{luc}\to K$ (note that since $K$ is compact, $K^{luc}=K$). As both $\widetilde{\alpha_G}$ and $\alpha_l$ are continuous extensions of $\alpha_G=\alpha_l|_G$ to $G^{luc}$, we must have $\alpha_l=\widetilde{\alpha_G}$. Hence, $\alpha_l$ is a homomorphism.
\end{proof}
\begin{corollary}\label{character}
If $G$ is an abelian locally compact group and $K$ is a compact group, then every continuous homomorphism $\alpha:G^*\to K$ has a unique continuous extension to $G^{luc}$.
\end{corollary}
\begin{proof}
 Let $\iota$ be any idempotent in $G^*$ and $(x_{\gamma})$ be a net in $G$ convergent to $\iota$. We have that for each $x\in G$
 $$x\iota=x(\lim_{\gamma}x_{\gamma})=\lim_{\gamma}(xx_{\gamma})=\lim_{\gamma}(x_{\gamma}x)=(\lim_{\gamma}x_{\gamma})x=\iota x.$$
 The result now follows from Proposition \ref{character0}.
\end{proof}
 We say a linear operator $T:C_0(G)^{\perp}\to C_0(H)^{\perp}$ is positive if for each positive linear functional $m\in M(G^*)=C(G^*)^*$ we have that $T(m)$ is a positive linear functional, here as usual $m$ positive means $m(f)\geq 0$, whenever $f\geq0$.\\
\begin{theorem}\label{extension}
Let $G$ and $H$ be locally compact groups with either $G$ or $H$ discrete and suppose that $T:C_0(G)^{\perp}\to C_0(H)^{\perp}$ is a weak-star continuous isometric isomorphism. If either $G$ is abelian, or $T$ is a positive operator,  then there exists a unique weak-star continuous isometric isomorphism $\tilde{T}:LUC(G)^*\to LUC(H)^*$ such that $\tilde{T}|_{C_0(G)^{\perp}}=T$.
\end{theorem}
 \begin{proof}
The uniqueness will follow from Proposition \ref{unique}. The proof of Corollary \ref{abc} shows that there exists a continuous character $\alpha:G^*\to\mathbb{T}$ and a topological isomorphism $\phi:G^*\to H^*$ such that $$T(\delta_x)=\alpha(x)\delta_{\phi(x)}\ \ \ (x\in G^*).$$ By Theorem \ref{nice}, there exists a topological isomorphism $\psi:G\to H$ such that $\tilde{\psi}|_{G^*}=\phi$. If $G$ is abelian, by Corollary \ref{character}, there exists a unique continuous character $\alpha_G:G\to\mathbb{T}$ such that $\tilde{\alpha_G}|_{G^*}=\alpha$; if $T$ is positive, $\alpha\equiv 1_{G^*}$ and $\alpha=\widetilde{1_{G}}|_{G^*}$. As noted in the introduction, $\tilde{T}=j_{\alpha_G,\psi}^*$ is a weak-star continuous isometric isomorphism of $LUC(G)^*$ onto $LUC(H)^*$. For $x\in G^*$, weak-star continuity and density considerations give $$\tilde{T}(\delta_x)=\alpha(x)\delta_{\phi(x)}=T(\delta_x)$$ and the proof is complete. 
 \end{proof}
 
 Suppose that $T:C_0(G)^{\perp}\to C_0(H)^{\perp}$ is an isometric isomorphism (not necessarily weak-star continuous). Then for each $x\in H$, the mapping 
 $$L_x:C_0(G)^{\perp}\to C_0(G)^{\perp},\ \  \text{where}
 \ \ \ L_x(m)=T^{-1}(\delta_xT(m))$$
 is an invertible isometric left multiplier on $C_0(G)^{\perp}$ (i.e. $L_x(mn)=L_x(m)n,\ m,n\in C_0(G)^{\perp}$) with inverse $L_{x^{-1}}$. Similarly, for each $x\in H$, the mapping 
 $$R_x:C_0(G)^{\perp}\to C_0(G)^{\perp},  \ \ \ \text{where}\ \ \  R_x(m)=T^{-1}(T(m)\delta_x)$$
 is an invertible isometric right multiplier on $C_0(G)^{\perp}$. We call $L_x:C_0(G)^{\perp}\to C_0(G)^{\perp}$ a left-point multiplier and $R_x:C_0(G)^{\perp}\to C_0(G)^{\perp}$ a right-point multiplier associated with $T$. The reader is referred to \cite[Sections 1.2.1-1.2.7]{MR1270014} for definitions and basic theorems regarding the left/right multipliers. \\
  
Suppose that $T:C_0(G)^{\perp}\to C_0(H)^{\perp}$ is a weak-star continuous isometric isomorphism. Moreover, suppose that $T=\tilde{T}|_{C_0(G)^{\perp}}$, where $\tilde{T}=j_{\alpha,\psi}^*:LUC(G)^*\to LUC(H)^*$ for some character $\alpha$ on $G$ and some topological isomorphism $\psi:G\to H$. (By Theorem \ref{extension}, this is the case when $G$ is discrete and either $G$ is abelian or $T$ is positive.) Then given $x\in H$, for each $m\in C_0(G)^{\perp}$
\begin{eqnarray*}
 L_x(m)=T^{-1}(\delta_xT(m))&=&\tilde{T}^{-1}(\delta_x\tilde{T}(m))\\
 &=&\tilde{T}^{-1}(\delta_x)\tilde{T}^{-1}(\tilde{T}(m))\\
 &=&\bar{\alpha}(\psi^{-1}(x))\delta_{\psi^{-1}(x)}m.
 \end{eqnarray*}
\noindent We shall say that a multiplier $L:C_0(G)^{\perp}\to C_0(G)^{\perp}$ is \textit{given by a point-mass} if there exist $y\in G$ and $\gamma\in\mathbb{T}$ such that $L(m)=\gamma\delta_ym$, for all $m\in C_0(G)^{\perp}$. The above argument shows that left-point multipliers associated with ``canonical form" isomorphisms $T=j_{\alpha,\psi}^*$ (and their inverses) are given by point-masses. We now prove the converse of this statement. 
\begin{theorem}\label{bcd}
Let $G$ and $H$ be locally compact groups and $T:C_0(G)^{\perp}\to C_0(H)^{\perp}$ be a weak-star continuous isometric isomorphism. Suppose that the left-point multipliers associated with $T$ and $T^{-1}$ are given by point-masses. Then $T$ takes the canonical form $T = j_{\beta, \gamma}^*$ for a character $\beta$ on $G$ and a topological isomorphism $\gamma: G\to H$. In particular, $T$ extends to a topological isomorphism of $LUC(G)^*$ onto $LUC(H)^*$ and $G$ and $H$ are topologically isomorphic.  
\end{theorem}
\begin{proof}
For each $x\in H$, let $\psi(x)\in G$ and $\alpha(x)\in\mathbb{T}$ be such that $L_x(m)=\alpha(x)\delta_{\psi(x)}m$, for all $m\in C_0(G)^{\perp}$. Suppose that $z\in G^*$ is right cancellable in $LUC(G)^*$. First we show that both $\alpha$ and $\psi$ are multiplicative. To see this, let $x,y\in H$. Then, we have that 
\begin{eqnarray*}
\alpha(xy)\delta_{\psi(xy)}\delta_z=L_{xy}(\delta_z)&=&T^{-1}(\delta_{xy}T(\delta_z))\\
&=&T^{-1}[\delta_xT\left(T^{-1}(\delta_yT(\delta_z))\right)]\\
&=&\alpha(x)\delta_{\psi(x)}T^{-1}(\delta_yT(\delta_z))\\
&=&\alpha(x)\alpha(y)\delta_{\psi(x)}\delta_{\psi(y)}\delta_z=\alpha(x)\alpha(y)\delta_{\psi(x)\psi(y)z},
\end{eqnarray*}
so $\alpha(xy)=\alpha(x)\alpha(y)$ and $\psi(xy)=\psi(x)\psi(y)$. Now we show that $\psi:H\to G$ is continuous. Suppose that $x_{\gamma}\to x$ in $H$. Since $T$ (and therefore $T^{-1}$) is weak-star continuous, we have that $L_{x_{\gamma}}(\delta_z)\to L_x(\delta_z)$ and so $\alpha(x_{\gamma})\delta_{\psi(x_{\gamma})}\delta_z\xrightarrow{w^*} \alpha(x)\delta_{\psi(x)}\delta_z$ in $C_0(G)^{\perp}$. Evaluating at the constant function $1$ in $LUC(G)$, we see that $\alpha(x_{\gamma})\to \alpha(x)$, and therefore $\delta_{\psi(x_{\gamma})}\delta_z\xrightarrow{w^*} \delta_{\psi(x)}\delta_z$ in $C_0(G)^{\perp}$. It follows from Lemma \ref{two parts} that $\delta_{\psi(x_{\gamma})}\xrightarrow{w^*}\delta_{\psi(x)}$. Thus, $\psi$ is continuous. To see that $\psi$ is a bijection, we note that given $y\in G$, $\gamma(y)\in H$ and $\beta(y)\in\mathbb{T}$ are such that 
$$\beta(y)\delta_{\gamma(y)}m=T(\delta_yT^{-1}(m))\ \ \ (m\in C_0(H)^{\perp}),$$ then the above argument shows that $\gamma$ and $\beta$ are also continuous maps. Letting $x\in G$ and $m=T^{-1}(\delta_z)\in C_0(G)^{\perp}$, we obtain 
\begin{eqnarray*}
\beta(\psi(x))\delta_{\gamma(\psi(x))z}=\beta(\psi(x))\delta_{\gamma(\psi(x))}\delta_z
&=&T(\delta_{\psi(x)}T^{-1}(\delta_z))\\
&=&T\left(\overline{\alpha(x)}\alpha(x)\delta_{\psi(x)}m\right)\\
&=&\overline{\alpha(x)}T(L_x(m))\\
&=&\overline{\alpha(x)}T\left( T^{-1}(\delta_xT(m))\right)\\
&=& \overline{\alpha(x)}\delta_x\delta_z=\overline{\alpha(x)}\delta_{xz}.
\end{eqnarray*}
 Hence $\gamma=\psi^{-1}$ (and $\beta\circ\psi=\bar{\alpha}$; equivalently $\beta=\bar{\alpha}\circ\psi^{-1}$). Now we show that $T$ has the desired canonical form. To see this, suppose that $p\in H^{luc}$. Let $(x_i)$ be a net in $H$ converging to $p$. We have that
\begin{eqnarray*}
\tilde{\alpha}(p)\delta_{\tilde{\psi}(p)}m&=&\lim_i\alpha(x_i)\delta_{\psi(x_i)}m\\
&=&\lim_iT^{-1}(\delta_{x_i}T(m))\\
&=&T^{-1}(\delta_pT(m))=T^{-1}(\delta_p)m,   
\end{eqnarray*}
for all $m\in C_0(G)^{\perp}$. Taking $m=z\in G^*$, a right cancellable element in $G^{luc}$, we have that $T^{-1}(\delta_p)=\tilde{\alpha}(p)\delta_{\tilde{\psi}(p)}$. So $T^{-1}|_{G^*}=j_{\alpha,\psi}^*|_{G^*}$ and, since $T^{-1}$ is weak-star continuous, we have that  $T^{-1}=j_{\alpha,\psi}^*|_{C_0(G)^{\perp}}$. It follows that $T$ has the canonical form $\displaystyle{T=j_{\bar{\alpha}\circ\psi^{-1},\psi^{-1}}^*|_{C_0(G)^{\perp}}=j_{\beta,\gamma}^*|_{C_0(G)^{\perp}}}$, and therefore extends to $LUC(G)^*$.
\end{proof}
It can be shown that if $G$ and $H$ are locally compact groups then any isometric (algebra) isomorphism $T:M(G)\to M(H)$ is weak-star continuous. It is thus natural to ask if every isometric isomorphism $T:LUC(G)^*\to LUC(H)^*$, equivalently $T:M(G^{luc})\to M(H^{luc})$, is also weak-star continuous. In fact, this question was the motivation for this paper. We have shown that any such isometric isomorphism maps $C_0(G)^{\perp}$ onto $C_0(H)^{\perp}$, so a related question is if every isometric isomorphism $T:C_0(G)^{\perp}\to C_0(H)^{\perp}$ is weak-star continuous. Similar questions can be asked about algebraic isomorphisms $\phi:G^{luc}\to H^{luc}$ and $\varphi:G^{*}\to H^{*}$.\\ 

We conclude with a discussion of the multipliers on $C_0(G)^{\perp}$, which may be useful with regards to the problems described above and which we think is of independent interest. Suppose that $G$ is a locally compact group. Since $LUC(G)^*$ is unital, a simple observation shows that every left multiplier $L$ is of the form $L(n)=m_0 n$:  
 $$L(n)=L(\delta_e n)=L(\delta_e)n.$$
 So, since $LUC(G)^*$ is unital, it is easy to characterize its left multipliers. The left multiplier $L$ is weak-star continuous if and only if $m_0\in M(G)$ (see \cite[Corollary 3]{MR817669}). Also, $L:LUC(G)^*\to LUC(G)^*$ is an invertible isometric left multiplier if and only if $m_0=\alpha\delta_x$, $\alpha\in\mathbb{T}$, $x\in G$ (see \cite[Corollary 1.2]{MR1005079}). So it is easy to characterize both the weak-star continuous and invertible isometric left multipliers on $LUC(G)^*$. Therefore, if we could show that every left-point multiplier on $C_0(G)^{\perp}$ associated with a weak-star continuous isometric isomorphism $T:C_0(G)^{\perp}\to C_0(G)^{\perp}$ extends to an isometric invertible left multiplier on $LUC(G)^*$, then every left-point multiplier on $C_0(G)^{\perp}$ must be given by a point-mass; by Theorem \ref{bcd} we could then conclude that $C_0(G)^{\perp}$ determines $G$. Similarly, every right multiplier $R$ on $LUC(G)^*$ is of the form $R(n)=nm_0$, for some $m_0\in LUC(G)^*$, and therefore, every right multiplier on $LUC(G)^*$ is weak-star continuous. As shown in Proposition \ref{continuous} below, the right-point multipliers $R_x$ associated with $T:C_0(G)^{\perp}\to C_0(H)^{\perp}$ (as defined above) are also always weak-star continuous  (without assuming that $T$ is weak-star continuous).  \\

 When $T:C_0(G)^{\perp}\to C_0(H)^{\perp}$ is an isomorphism and $x\in H$, it is readily verified that the pair $(L_x,R_x)$ is a double centralizer of $C_0(G)^{\perp}$ (i.e. $mL_x(n)=R_x(m)n$ for all $m,n\in C_0(G)^{\perp}$). We note that if $(L,R)$ is a double centralizer of $C_0(G)^{\perp}$ such that $L(n)=m_0n$, for some $m_0\in  LUC(G)^*$, then using a right cancellable element $z\in C_0(G)^{\perp}$ - see Lemma \ref{two parts} - we have that $R(n)=nm_0$:
 $$R(n)z=nL(z)=nm_0z.$$

 \begin{prop}\label{continuous}
Suppose that $G$ and $H$ are locally compact groups and $T:C_0(G)^{\perp}\to C_0(H)^{\perp}$ is an isometric isomorphism. Then for each $x\in H$, the isometric right multiplier $R_x(m)=T^{-1}(T(m)\delta_x)$ is weak-star continuous.
\end{prop}
\begin{proof}
It is enough to show that $R_x$ is weak-star continuous on the unit ball of $C_0(G)^{\perp}$. Suppose that $n_i\xrightarrow{w^*} n$ in $C_0(G)^{\perp}$, with $(n_i)$ bounded in norm by 1. Clearly $(R_x(n_i))$ is also bounded by $1$. Suppose that $z\in G^*$, is right cancellable in $G^{luc}$. We have that 
\begin{align*}
 \nonumber R_x(n_i)\delta_z&=n_iL_x(\delta_z) \\
   \nonumber &  \xrightarrow{w^*} n L_x(\delta_z)= R_x(n)\delta_z.
     \end{align*}
     By Lemma \ref{two parts}, $R_x(n_i)\to R_x(n)$ weak-star in $C_0(G)^{\perp}$, as needed.
\end{proof}
 
 Another related interesting question is whether there is a characterization for the multipliers on $C_0(G)^{\perp}$. Unlike, $LUC(G)^*$, $C_0(G)^{\perp}$ does not possess a unit or even an approximate identity. In fact, it is not hard to see that the annihilator of $C_0(G)^{\perp}$ in $LUC(G)^*$ is zero. Noting that $C_0(G)^{\perp}$ is a closed ideal in $LUC(G)^*$, there is a natural embedding of $LUC(G)^*$ into $\mathcal{D}(C_0(G)^{\perp})$, the double centralizer algebra of $C_0(G)^{\perp}$. Does  $LUC(G)^*=\mathcal{D}(C_0(G)^{\perp})$? Is there any nice characterization for $\mathcal{D}(C_0(G)^{\perp})$? 
\end{section}
\section*{Acknowledgement} 
I would like to express my gratitude to my PhD advisor, Ross Stokke, for suggesting the problem that motivated the results in this article. Also, I would like to thank my PhD advisors, Fereidoun Ghahramani and Ross Stokke, for their valuable comments.

 \nocite{*}
\bibliographystyle{alpha}

\bibliography{mybib}
\end{document}